\documentclass[11pt]{amsart}
\usepackage{amsmath}
\usepackage{amssymb}
\usepackage{graphicx}
\usepackage{booktabs, multicol, multirow}
\usepackage{bigstrut}
\usepackage{array}
\usepackage{tabularx}
\usepackage{pdflscape}
\usepackage{adjustbox}
\usepackage{longtable}
\usepackage{cite}
\newtheorem{theorem}{Theorem}[section]

\newtheorem{lemma}[theorem]{Lemma}

\theoremstyle{definition}

\numberwithin{equation}{section}
\title[abbreviated title ]{Iterative algorithm with structured diagonal Hessian approximation for solving nonlinear least squares problems}
\author[A.M. Awwal]{Aliyu Muhammed Awwal$^{\rm 1,3}$}
\author[P. Kumam]{Poom Kumam$^{\rm 1,2*}$}
\author[H. Mohammad]{Hassan Mohammad$^{\rm 4}$}
\address{$^{\rm 1}$KMUTTFixed Point Research Laboratory, Department of Mathematics, Room SCL 802 Fixed Point Laboratory, Science Laboratory Building, Faculty of Science, King Mongkut's University of Technology Thonburi (KMUTT), 126 Pracha-Uthit Road, Bang Mod, Thrung Khru, Bangkok 10140, Thailand}
\address{$^{\rm 2}$KMUTT-Fixed Point Theory and Applications Research Group, Theoretical and Computational Science Center (TaCS), Science Laboratory Building, Faculty of Science, King Mongkut's University of Technology Thonburi (KMUTT), 126 Pracha-Uthit Road, Bang Mod, Thrung Khru, Bangkok 10140, Thailand}
\address{$^{\rm 3}$Department of Mathematics, Faculty of Science, Gombe State University, Gombe, Nigeria}
\address{$^{\rm 4}$Department of Mathematical Sciences, Faculty of Physical Sciences, Bayero University, Kano.  Kano, Nigeria}
\email[A.M. Awwal]{{\tt aliyumagsu@gmail.com}}
\email[P. Kumam]{\tt poom.kumam@mail.kmutt.ac.th}
\email[H. Mohammad]{{\tt hmuhd.mth@buk.edu.ng}}
\begin{document}
\begin{abstract}
Nonlinear least squares problems are special class of unconstrained optimization problems in which their gradient and Hessian have special structures. In this paper, we exploit these structures and proposed a matrix free algorithm with diagonal Hessian approximation for solving nonlinear least squares problems. We devise appropriate safeguarding strategies to ensure the Hessian matrix is positive definite throughout the iteration process. The proposed algorithm generates descent direction and is globally convergent. Preliminary numerical experiments shows that the proposed method is competitive with a recent developed similar methods.
\end{abstract}
\maketitle
\section{Introduction}
\label{intro}
\indent In this paper, we consider nonlinear least squares problems, the special class of unconstrained optimization problems, of the form
\begin{equation}\label{prob1}
\min_{x\in \mathbb{R}^n}f(x),~~~f(x)=\frac{1}{2}\|F(x)\|^2,
\end{equation}
where $F(x)=(F_1(x),\cdots, F_m(x))^T$ and each residual $F_i : \mathbb{R}^n \to \mathbb{R},$ $i=1,\cdots, m$ (usually $m\geq n$), is twice continuously differentiable function.  Let $J(x) \in \mathbb{R}^{m\times n}$ denotes Jacobian of the residual function $F(x)$ and $g(x),$ denote the gradient of the objective function $f$, $\nabla f(x_k)$, and $H(x)$ denote the Hessian of the objective function $\nabla^2 f(x)$. The gradient and Hessian of problem (\ref{prob1}) have special structures and are respectively given by
\begin{equation}\label{gradient}
g(x):= \sum_{i=1}^m F_i(x)\nabla F_i(x) = J(x)^TF(x)
\end{equation}
\begin{equation}\label{hessian}
H(x):= \sum_{i=1}^m \nabla F_i(x)\nabla F_i(x)^T + \sum_{i=1}^m F_i(x)\nabla ^2 F_i(x)=J(x)^TJ(x)+C(x),
\end{equation}
where $F_i$ is the $i$th component of $F$, $\nabla ^2 F_i(x)$ is its Hessian, and $C(x)$ is a square matrix representing the second term of  the Hessian.

Nonlinear least squares problems have been studied extensively, and many iterative algorithms for solving them have been proposed. These generally fall into two categories, namely, general unconstrained optimization algorithms that includes Newton's method and quasi-Newton methods; and special methods, that take the special structure of the problem into account, which constitute Gauss-Newton method, Levenberg-Marquardt method and Structured quasi-Newton methods (see \cite{xu1990hybrid,zhang2012local,zhang2010derivative,yuan2009subspace,fletcher1987hybrid}). For a brief survey of methods for addressing nonlinear least squares problems, interested reader may refer to the recent articles by Mohammad et al. \cite{mohammad2019survey} and Yuan \cite{yuan2011recent}. 

The study of efficient algorithm for nonlinear least squares (NLS) problems is important because of its numerous areas of applications such as data fitting, optimal control, parameter estimation, experimental design, data assimilation, and imaging problems (see \cite{golub2003separable,kim2007interior,li2012maximum,cornelio2011regularized,barz2015nonlinear,tang2011regularization}). For instance, it is often common to measure the discrepancy between a proposed parametrized model and the observed behavior of a given system. To select values for the parameters that best match the model to the data, it is usual to minimize the sum of the squares of the residuals $F_i^{'s}$. 
The special structure of problem (\ref{prob1}) can always be explored to devise efficient algorithms for obtaining its solution. For example Kobayashi et al. \cite{kobayashi2010nonlinear} exploits the structure of the nonlinear least squares by introducing a class of matrix-free structured methods that falls into the category of conjugate gradient algorithms with modified secant condition for solving nonlinear least squares problems. Motivated by their idea, Dehghani and Mahdavi-Amiri \cite{dehghani2018scaled} proposed a modified secant relation specifically to get more information of the Hessian of the nonlinear least squares objective function. Furthermore, they proposed another class of conjugate gradient methods for addressing nonlinear least squares problems. In another attempt, but different approach, Mohammad and Waziri \cite{mohammad2019structured} proposed two structured Barzilai-Borwein step sizes for solving nonlinear least squares. 

Recently, Mohammad and Sandra \cite{mohammad2018structured} proposed a diagonal Hessian approximation method for nonlinear least squares problems in which the diagonal approximation of the Hessian of the objective function  is obtained using a structured secant condition that  have some information of the exact Hessian. However, as a final remarks, the authors comment on the need for further research that investigate a better approximation of the Hessian matrix that involved its special structure. 

We feel that approximating the first and second terms of the Hessian matrix (\ref{hessian}) will lead to substantial lost of information about the Hessian. In this paper, we proposed a diagonal Hessian with better approximation by exploiting the special structure of problem (\ref{prob1}) and obtained a matrix-free algorithm. The main difference between our method and the method in \cite{mohammad2018structured}, is that, in building our diagonal matrix, we take the whole information of the first term of (\ref{hessian}) and approximate its second term. By this, our diagonal matrix contains more information than the one proposed in \cite{mohammad2018structured}. Our proposed method generates descent directions and is globally convergent.

The remaining of this paper is organized as follows. In Section 2, we present the proposed method and its algorithm. The convergence analysis is discussed in Section 3 and in Section 4, we give numerical experiments. Throughout this article, we use the following notations for the objective function$ f(x_k)=f_k,\text{ for the residual }~ F(x_k)=F_k,\text{ for any matrix}~ A(x_k)=A_k,~\text{ and }\|\cdot\|$ for the Euclidean norm of vectors and the induced 2-norm of matrices.
\section{Proposed method}
\label{sec:1}
An important concept of a structured quasi-Newton method for nonlinear least squares is the structure principle \cite{dennis1989convergence}. Now, we provide the necessary elements to develop our proposed diagonal Hessian approximation, taking the special structures of the gradient and Hessian of the problem (\ref{prob1}) into account.\\
Consider the second term of the Hessian matrix (\ref{hessian}). Suppose that at certain iteration $k-1,~~k\ge 1$ the second term of the Hessian matrix (\ref{hessian}) is 
\begin{equation}
C(x_{k-1})=\sum_{i=1}^m F_i(x_{k-1})\nabla ^2 F_i(x_{k-1}),
\end{equation}
so that the updating matrix $C(x_k)$ which satisfies the secant equation $$C(x_k)s_{k-1}=y_{k-1},$$ can be obtained as follows.\\
The Taylor's expansion of $\nabla F_i(x_{k-1})$ can be written as
\begin{equation}\label{taylor1}
\nabla F_i(x_{k-1})=\nabla F_i(x_k) + \nabla^2 F_i(x_k)^T(x_{k-1}-x_k) + \textit{\textbf{o}}(\|x_{k-1}-x_k\|),
\end{equation}
where $\textit{\textbf{o}}:\mathbb{R}_+\to \mathbb{R}^n$ such that for each $i=1,\cdots,n,$ $\lim\limits_{\xi\to 0}\frac{o^i(\xi)}{\xi}=0.$\\
Let $s_{k-1}=x_k-x_{k-1}.$ Multiplying (\ref{taylor1}) by $F_i(x_k)$ and rearranging, we have
\begin{equation}\label{taylor2}
F_i(x_k)\nabla^2 F_i(x_k)^Ts_{k-1}=F_i(x_k)\nabla F_i(x_k) - F_i(x_k)\nabla F_i(x_{k-1}) + F_i(x_k)\textit{\textbf{o}}(\|s_{k-1}\|),
\end{equation}
By summing both sides of (\ref{taylor2}) for $i=1,\cdots,m$ we obtain
\begin{equation}\label{Csy}
C(x_k)s_{k-1}=J(x_k)^TF(x_k)-J(x_{k-1})^TF(x_k)+(F_k^T\textbf{1}_m)\textit{\textbf{o}}(\|s_{k-1}\|).
\end{equation}
Putting (\ref{Csy}) into (\ref{hessian}) implies
\begin{equation}\label{Hsy}
H_ks_{k-1}=J_k^TJ_ks_{k-1}+(J_k-J_{k-1})^TF(x_k)+(F_k^T\textbf{1}_m)\textit{\textbf{o}}(\|s_{k-1}\|).
\end{equation}

Let $D_k\approx H_k$ such that $D_k$ is a diagonal matrix approximately satisfying the secant equation
\begin{equation}\label{Csyapprx}
D_ks_{k-1}\approx y_{k-1},
\end{equation}
where $y_{k-1}=J_k^TJ_ks_{k-1}+(J_k-J_{k-1})^TF(x_k).$
For convenience, we denote the first and second terms of $y_{k-1}$ as   
\begin{equation}\label{yapprx}
\hat{y}_{k-1}=J_k^TJ_ks_{k-1} ~~ \text{and} ~~ \overline{y}_{k-1}=(J_k-J_{k-1})^TF(x_k).
\end{equation}
The following Lemma comes from \cite{mohammad2018structured} and will be useful in defining the entries of the diagonal matrix $D_k$ in view of the secant equation (\ref{Csyapprx}).
\begin{lemma}\label{lem1}
Let Let $D=\text{diag}(d)$ be a diagonal matrix in $\mathbb{R}^{n\times n},$ and let $c$ and $s$ be vectors in $\mathbb{R}^n.$ Then, the solution of the constrained linear least-squares problem with simple bounds:
\begin{equation}
\min_{d\in \mathbb{R}^n}\frac{1}{2} \|diag(d)s-c\|^2
\end{equation}
$$ subject~~ to ~-d\leq 0$$
is given by 
\begin{equation}\label{dcsi}
 {d^i}=
\begin{cases}
\frac{c^i}{s^i}, & \text{if} ~~ \frac{c^i}{s^i}>0,\\
0, & \text{if} ~~ \frac{c^i}{s^i}\leq 0, ~ \text{or} ~ s^i=0,
\end{cases}
i=1,\cdots,n.
\end{equation}
\end{lemma}

Let the diagonal matrix $D_k$ be decompose into two diagonal matrices i.e. $D_k=diag(a_k)+diag(b_k)$, where $a_k, ~b_k$ are vectors representing the first and second terms of $y_{k-1}$ respectively. For the diagonal matrix to be positive definite, all the diagonal entries $a_k^i$ and $b_k^i$ for $i=1,\cdots, n,$ must be strictly positive. Now, applying Lemma \ref{lem1} to the secant equation (\ref{Csyapprx}) for $i=1,\cdots,n,$ Equation (\ref{dcsi}) becomes 
\begin{equation}\label{aysi}
 {a_k^i}=
\begin{cases}
\frac{\hat{y}_{k-1}^i}{s_{k-1}^i}, & \text{if} ~~ \frac{\hat{y}_{k-1}^i}{s_{k-1}^i}>0,\\
0, & \text{if} ~~ \frac{\hat{y}_{k-1}^i}{s_{k-1}^i}\leq 0, ~ \text{or} ~ s^i=0,
\end{cases}
i=1,\cdots,n.
\end{equation}
and
\begin{equation}\label{bysi}
 {b_k^i}=
\begin{cases}
\frac{\overline{y}_{k-1}^i}{s_{k-1}^i}, & \text{if} ~~ \frac{\overline{y}_{k-1}^i}{s_{k-1}^i}>0,\\
0, & \text{if} ~~ \frac{\overline{y}_{k-1}^i}{s_{k-1}^i}\leq 0, ~ \text{or} ~ s^i=0,
\end{cases}
i=1,\cdots,n.
\end{equation}

From the above Lemma \ref{lem1}, the diagonal matrix has nonnegative diagonal entries which means it is positive semidefinite matrix. In order to ensure the diagonal matrix is positive definite, in the next subsection we provide a safeguard strategy similar to the one given in \cite{mohammad2018structured}, that guarantee each diagonal entries $a_k^i$ and $b_k^i$ for $i=1,\cdots n,$ is strictly positive.
\subsection{Safeguarding strategy}
We consider the situation in which $\hat{y}_{k-1}^i,s_{k-1}^i$ and $\overline{y}_{k-1}^i,s_{k-1}^i$ have different signs with $s_{k-1}^i\neq 0.$ Let $\gamma\in(0,1)$ be a shrinking parameter and $\rho>0$ be a tolerance for ensuring strictly values.
\begin{itemize}
\item[\textbf{Case (a)}]
\begin{itemize}
\item[] Suppose $s_{k-1}^i>0$.\\
\item [\textbf{(ai)}]  If $\hat{y}_{k-1}^i\leq 0,$ i.e., $(J_k^TJ_ks_{k-1})^i\leq 0.$ Then redefine $\hat{y}_{k-1}^i$ as
\begin{equation}\label{s1}
\hat{y}_{k-1}^i=\gamma\max\{|(J_k^TJ_ks_{k-1})^i|,\rho\}
\end{equation}
so that $a_k^i=\frac{\hat{y}_{k-1}^i}{s_{k-1}^i}>0.$
\item[\textbf{(aii)}] If $\overline{y}_{k-1}^i\leq 0,$ i.e., $(J_k^TF_k)^i\leq (J_{k-1}^TF_k)^i.$ Then redefine $\overline{y}_{k-1}^i$ as
\begin{equation}\label{s2}
\overline{y}_{k-1}^i=\gamma\left\{\max\left\{\max|(J_k^TF_k)^i|,|(J_{k-1}^TF_k)^i|\right\},\rho\right\}
\end{equation}
so that $b_k^i=\frac{\overline{y}_{k-1}^i}{s_{k-1}^i}>0.$
\end{itemize}
\item[\textbf{Case(b)}]
\begin{itemize}
\item[] Suppose $s_{k-1}^i<0$.
\item [\textbf{(bi)}] If $\hat{y}_{k-1}^i\geq 0,$ i.e., $(J_k^TJ_ks_{k-1})^i\geq 0.$ Then redefine $\hat{y}_{k-1}^i$ as
\begin{equation}\label{s3}
\hat{y}_{k-1}^i=-\gamma\max\{(J_k^TJ_ks_{k-1})^i,\rho\}
\end{equation}
so that $a_k^i=\frac{\hat{y}_{k-1}^i}{s_{k-1}^i}>0.$
\item[\textbf{(bii)}] If $\overline{y}_{k-1}^i\geq 0,$ i.e., $(J_k^TF_k)^i\geq (J_{k-1}^TF_k)^i.$ Then redefine $\overline{y}_{k-1}^i$ as
\begin{equation}\label{s4}
\overline{y}_{k-1}^i=-\gamma\left\{\max\left\{\max|(J_k^TF_k)^i|,|(J_{k-1}^TF_k)^i|\right\},\rho\right\}
\end{equation}
so that $b_k^i=\frac{\overline{y}_{k-1}^i}{s_{k-1}^i}>0.$
\end{itemize}
\end{itemize}
In a situation where $s_{k-1}^i=0,$ then $a_{k-1}^i$ and $b_{k-1}^i$ will assume any suitable nonnegative safeguarding value.
\subsection{Algorithm}
In this subsection, we present the proposed algorithm. Let $d_k$ and $g_k$ denote the search direction and the gradient of the objective function (\ref{prob1}) respectively. The search direction $d_k$, is obtained by solving the linear systems $$H_kd_k=-g_k,$$ where 
\begin{equation}
H_k=
\begin{cases}
I, & \text{if}~~ k=0\\
diag(a_k+b_k), & \text{if}~~ k\geq 1
\end{cases}
\end{equation}
is a diagonal matrix whose entries are computed by
\begin{equation}
a_k^i + b_k^i =
\begin{cases}
\frac{\hat{y}_{k-1}^i+\overline{y}_{k-1}^i}{s_{k-1}^i} & \text{if}~~ s_{k-1}^i \neq 0\\
1 & \text{if} ~~ s_{k-1}^i=0,
\end{cases} 
i=1,\cdots,n.
\end{equation}
The vectors $\hat{y}_{k-1}$ and $\overline{y}_{k-1}$ are defined by (\ref{yapprx}) with some of their components possibly redefined by (\ref{s1})$-$(\ref{s4}). Furthermore, we safeguard the diagonal entries of the diagonal Hessian $H_k$ from assuming extremely small and extremely large values by means of projecting them into a given scalar interval $[l,u],$ such that $0<l\leq 1 \leq u<<+\infty.$ Hence, the $i$th diagonal entry of our Hessian matrix $H_k$ in which $y_{k-1}^i=\hat{y}_{k-1}^i+\overline{y}_{k-1}^i$ for each $i=1,2,...,n$ is given by 
\begin{equation}\label{dhessian}
 {h_k^i}=
\begin{cases}
\frac{y_{k-1}^i}{s_{k-1}^i}, & \text{if} ~~ l\leq \frac{y_{k-1}^i}{s_{k-1}^i}\leq u,\\
l, & \text{if}~~ \frac{y_{k-1}^i}{s_{k-1}^i}< l,\\
u, & \text{if}~~ \frac{y_{k-1}^i}{s_{k-1}^i}> u,\\
1, & \text{if} ~~ s_{k-1}^i=0
\end{cases}
\end{equation}
Given a starting point $x_0,$ we compute the next iterate via
\begin{equation}\label{xkdk}
x_{k+1}=x_k +\alpha_kd_k, ~~ k=0,1,2,\cdots.
\end{equation}
Here the search direction is given by
\begin{equation}\label{direction}
d_k=H_k^{-1}g_k,
\end{equation}
$H_k$ is a structured diagonal Hessian with diagonal entries $h_k^i$, the gradient $g_k=J_k^TF_k$ and $J_k$ and $F_k$ are the Jacobian matrix and function evaluation at $x_k$ respectively.\\
We adopt the non-monotone line search proposed by Zhang and Hager \cite{zhang2004nonmonotone} to determine the step length $\alpha_k$. Let the search direction $d_k$ defined by (\ref{direction}) be a descent direction, then the step length $\alpha_k>0$ in (\ref{xkdk}) should satisfy the following non-monotone Armijo-type line search technique
\begin{equation}\label{linesearch}
f(x_k+\alpha_kd_k)\leq P_k + \theta \alpha_k g_k^T d_k,
\end{equation}
where 
\begin{equation}\label{PQ}
\begin{cases}
P_0=f(x_0)\\
P_{k+1}=\frac{\eta_kQ_kP_k+f(x_{k+1})}{Q_{k+1}}\\
Q_0=1,\\
Q_{k+1}=\eta_kQ_k+1,
\end{cases}
\end{equation}
and $\theta\in(0,1),$ $\eta_k\in[0,1].$\\
Next we give the following remarks:
\begin{itemize}
\item[\textbf{Remark A}]
\item[\textbf{\textit{(i)}}] Note that the $P_{k+1}$ in the above line search technique is a convex combination of $P_k$ and $f(x_{k+1}).$ Since $P_0=f(x_0)$, it follows that the sequence $\{P_k\}$ is a convex combination of the function values $f(x_i),$ for $i=0,1,2,\cdots,k.$
\item[\textbf{\textit{(ii)}}] The parameter $\eta_k$ controls the degree of monotonicity. If for each $k,$ $\eta_k=0,$ then the line search (\ref{linesearch}) is the usual monotone (Armijo-type); otherwise, it is non-monotone.
\item[\textbf{\textit{(iii)}}] If for each $k,$ $\eta_k=1,$ then $P_k=\psi_k$ where
\begin{equation}\label{etaeqaul1}
\psi_k= \frac{1}{k+1}\sum_{i=1}^kf(x_i)
\end{equation}
\end{itemize}
We now formally state the steps of our proposed iterative algorithm with structured diagonal Hessian.\\
\begin{itemize}
\item[]\textbf{Algorithm 1: Algorithm with Structured Diagonal Hessian (ASDH)}
\item[\textbf{Step 0.}]  Given $x_0\in \mathbb{R}^n$, $\gamma, \theta \in(0,1)$, $0\leq\eta_{\min} \leq \eta_{\max} \leq 1,$ $0<l \leq 1 \leq u,$ $\rho,\epsilon >0$, and $k_{\max} \in \mathbb{N}.$\\[5pt]
\item[\textbf{Step 1.}] Set $k=0,$ $H_k=I,$ $Q_k=1$. Compute $F_k$ and $g_k,$ and set $P_k=f_k.$\\[5pt]
\item[\textbf{Step 2.}] If $\|g_k\|\leq\epsilon$ and $k\geq k_{\max},$ stop. Else compute $d_k$ using (\ref{direction}).\\[5pt]
\item[\textbf{Step 3.}] Perform nonmonotone line search\\[2pt]
\begin{itemize}
\item[\textbf{Step 3.1.}] Set $ \alpha = 1,$\\[2.5pt]
\item[\textbf{Step 3.2.}] if the following inequality
$$f(x_k+\alpha d_k)\leq C_k +\alpha\theta g_k^Td_k.$$
holds, then proceed to \textbf{Step 4.}\\ Else, set $\alpha_k=\alpha/2$ and repeat \textbf{Step 3.2}\\[5pt]
\end{itemize}
\item[\textbf{Step 4.}] Set $\alpha_k=\alpha$ and compute the next iterate using (\ref{xkdk}).\\[5pt]
\item[\textbf{Step 5.}] Set $s_k=\alpha_kd_k$ and compute $\hat{y}_k$ and $\overline{y}_k$ using (\ref{yapprx}).\\[5pt]
\item[\textbf{Step 6.}] Safeguard $\hat{y}_k$ and $\overline{y}_k$ using (\ref{s1})$-$(\ref{s4}) in case any of them has different sign with $s_k^i$.\\[5pt]
\item[\textbf{Step 7.}] Update the diagonal Hessian $H_{k+1}$ using (\ref{dhessian}).\\[5pt]
\item[\textbf{Step 8.}] Choose $\eta_k\in[\eta_{\min},\eta_{\max}]$ and compute $Q_{k+1}$ and $P_{k+1}$ using (\ref{PQ}).\\[5pt]
\item[\textbf{Step 9.}] Set $k=k+1$ and go to \textbf{step 2.}
\end{itemize}
\begin{itemize}
\item[\textbf{Remark B}]
\item[\textbf{\textit{(i)}}] Though the vectors $g_k,$ $\hat{y}_k$ and $\overline{y}_k$  are in the form of matrix-vector products, these products were obtained by writing a MATLAB code that computes them directly without forming or storing the Jacobian matrix.
\item[\textbf{\textit{(ii)}}] Since the Hessian $H_k$ is a diagonal matrix, its inverse $H_k^{-1}$ is obtained by taking the reciprocal of each $h_i$ $(i=1,2,\cdots,n)$ using Equation (\ref{dhessian}) taking into account our safeguarding rule. Therefore, the product $H_k^{-1}g_k$ is simply component-wise vector multiplications.
\end{itemize}
From the discussions in \textit{Remark B} above, we can see that our proposed method is matrix-free and therefore suitable for large-scale problems. 

\section{Convergence Analysis}
In this section, we discuss the global convergence of our proposed method. We begin by stating the following assumptions which will be useful in our analysis.\\
\begin{itemize}
\item[\bf{A1.}] The level set $\mathcal{D}=\{x\in \mathbb{R}^n ~|~ f(x)\leq f(x_0)\}$ is bounded, i.e. there exists a positive constant $\omega$ such that $\|x\|\leq \omega$ for all $x\in\mathcal{D}.$
\item[\bf{A2.}] There exist constants $L_1$ and $L_2$ such that for all $x,y\in \mathcal{D},$ we have
\begin{equation}\label{JL}
\|J(x)-J(y)\|\leq L_1\|x-y\|
\end{equation}
and
\begin{equation}\label{FL}
\|F(x)-F(y)\|\leq L_2 \|x-y\|
\end{equation}
From (\ref{JL}) and (\ref{FL}), we can obtain the followings
$$\|g(x)-g(y)\|\leq l\|x-y\|,$$
$$\|F(x)\|\leq \omega_1, ~~ \|J(x)\|\leq \omega_2, ~~ \|g(x)\|\leq \gamma_3,$$
where $l,$ $\omega_1,$ $\omega_2$ and $\omega_3$ are positive constants.
\end{itemize}
\begin{lemma}\label{lem2}
Let the sequence of search directions $\{d_k\}$ be generated by the ASDH algorithm, then there exist $m_1, m_2$ positive constants such that for all $k=0,1,2,\cdots,$ the following relations hold
\begin{equation}\label{descentdirection}
g_k^Td_k \leq -m_1\|g_k\|^2,
\end{equation}
and
\begin{equation}\label{directionbdd}
\|d_k\| \leq m_2\|g_k\|.
\end{equation}
\end{lemma}
\begin{proof}
From the definition of $H_k^{-1},$ the diagonal entries defined by (\ref{dhessian}) is bounded for each $i$ and for all $k,$ i.e. $l\leq h_k^i \leq u.$\\ Now, by the Equation (\ref{direction}) we have
\begin{equation*}
\begin{split}
 g_k^Td_k & = -g_k^TH_k^{-1}g_k\\
 & = -\sum_{i=1}^n(g_k^i)^2/h_k^i\\
 & \leq - (1/u)\sum_{i=1}^n (g_k^i)^2\\
 & = - (1/u) \|g_k\|^2.
 \end{split}
\end{equation*}
If we let $m_1 = 1/u,$ then (\ref{descentdirection}) holds.\\
In a similar way, since the diagonal matrix is always symmetric, we get
\begin{equation*}
\begin{split}
 \|d_k\|^2 & = g_k^TH_k^{-2}g_k\\
 & = \sum_{i=1}^n(g_k^i/h_k^i)^2\\
 & \leq (1/l^2)\sum_{i=1}^n (g_k^i)^2\\
 & = (1/l^2) \|g_k\|^2.
 \end{split}
\end{equation*}
If we let $m_2 = 1/l,$ we obtain (\ref{directionbdd}) and the proof is complete.
\end{proof}
Equations (\ref{descentdirection}), (\ref{directionbdd}) and the Proposition 1 in \cite{mohammad2018structured} imply tha the ASDH Algorithm is well-defined. Next, we state the following result which comes from Lemma 1.1 in \cite{zhang2004nonmonotone}. 
\begin{lemma}\label{lem3}
The iterates generated by the ASDH Algorithm satisfy $f(x_k) \leq P_k \leq \psi_k,$ for all $k\geq 0,$ where $P_k$ and $\psi_k$ are defined by (\ref{PQ}) and (\ref{etaeqaul1}) respectively.
\end{lemma}
\begin{lemma}\label{lem4}
The sequence of iterates $\{x_k\}$ generated by the ASDH Algorithm satisfy $f(x_k)\leq f(x_0),$ for each $k\geq 0.$
\end{lemma}
\begin{proof}
By substituting $\eta_kQ_k=Q_{k+1}-1,$ in $P_{k+1}$ defined in Equation (\ref{PQ}) we have
\begin{equation*}
\begin{split}
P_{k+1} &=\frac{(Q_{k+1}-1)P_k + f(x_{k+1})}{Q_{k+1}},\\
&\leq  \frac{Q_{k+1}P_k-P_k+P_k+\alpha_k\theta m_1g_k^Td_k}{Q_{k+1}},\\
&\leq \frac{Q_{k+1}P_k-\alpha_k\theta m_1 \|g_k\|^2}{Q_{k+1}},\\
&=P_k-\frac{\alpha_k\theta m_1 \|g_k\|^2}{Q_{k+1}}\\
& \leq P_k.
\end{split}
\end{equation*}
The first two inequalities respectively come from (\ref{linesearch}) and (\ref{descentdirection}). 
Now, since $f_k \leq P_k,$ (Lemma \ref{lem3}), we obtain
\begin{equation*}
f(x_{k+1})\leq P_{k+1}\leq P_k \leq P_{k-1} \leq \cdots \leq P_0 = f(x_0),
\end{equation*}
so that $\{x_k\}\subset\mathcal{D}$ and the proof is complete.
\end{proof}
\begin{theorem}
Let $f(x)$ be defined by problem (\ref{prob1}) and suppose assumptions A1 and A2 hold. Then, the sequence of iterates $\{x_k\}$ generated by the ASDH Algorithm is contained in the level set $\mathcal{D}$ and
\begin{equation}
\liminf_{k\to \infty} \|g_k\| =0.
\end{equation}
Moreover, if $\eta_{\max}<1$ then 
\begin{equation}
\lim_{k\to \infty} \|g_k\| =0.
\end{equation}
\end{theorem}
\begin{proof}
The proof of this theorem follows directly from \cite{zhang2004nonmonotone}.
\end{proof}
\section{Numerical Experiments}
In this section, we turn our attention to numerical experiments to assess the performance of our proposed ASDH method compared to SDHAM method proposed in \cite{mohammad2018structured}. The SDHAM method is also a matrix-free algorithm for nonlinear least squares problems which exploits the special structure of the Hessian of the objective function; and generates its search direction using diagonal Hessian approximation similar to our proposed ASDH algorithm.

In our experiment, we solved 30 test problems of which 22 are large scale and 8 are small scale (see Table \ref{testproblem}). We vary the dimensions of the large scale problems as 1000, 5000; 10,000. All the test problems considered are properly cited. The parameters used in the experiment are as follows
\begin{itemize}
    \item ASDH algorithm: $\gamma=0.2,$ $\eta_k=0.75e^{(-(k/45)^2)}+0.1 ~\text{with}~ \eta_{\min}=0.1,$ $\eta_{\max}=0.85,$ $l=10^{-30},$ $u=10^{30};$ and $\rho=0.0001$\\
    \item SDHAM algorithm: All parameters are as presented in \cite{mohammad2018structured}.
\end{itemize}
All codes were written in MATLAB R2017a and run on a PC with intel COREi5 processor with 4GB of RAM and CPU 2.3GHZ speed. The iteration is terminated whenever the inequality $\|g_k\|\leq 10^{-4}$ is satisfied. Failure, denoted by F, is recorded when the number of iterations exceeds 1,000 and the stopping criterion mentioned above has not been satisfied.

In Tables \ref{table1}$-$\ref{table4}, we report the results of the following information: the number of iterations (NITER)
needed by each solver to converge to an approximate solution, the number of function evaluation (NFVAL), the number of matrix-vector product (NMVP) the CPU time in seconds (TIME), and the objective function $f$ value at the minimizer (FVALUE). The test problems are denoted by P$i,$ $i=1,2,\cdots,30.$

In the Figures \ref{ITER}$-$\ref{TIME}, we adopt the popular performance profile by Dolan and Mor{\'e} \cite{dolan2002benchmarking} to compare the performance of the ASDH method with that of SDHAM method based on the number of iterations, number of functions evaluation, number of matrix-vector product and CPU time. Though the two methods are competitive, it can be seen from the Figures \ref{ITER}$-$\ref{TIME} that all the curves with respect to our proposed ASDH method stay longer on the vertical axis which means it solves more problems with less NITER, NFVAL, NMVP and TIME compared to the SDHAM method. Specifically, it can observed from Figures \ref{ITER}$-$\ref{GVAL} that our method solves about $80\%$ of the test problems with least NITER, NFVAL, NMVP. Also, from Figure \ref{TIME} we can see that our method solves about $70\%$ of the test problems with least CPU TIME.\\
Moreover, from the information reported in Table \ref{table1}$-$\ref{table4}, it can be seen that ASDH method solves all the test problems without any failure while the SDHAM recorded 3 failures. In the overall experiments, ASDH method needs less number of iterations, number of functions evaluation, number of matrix-vector product and CPU time to obtain the minimizer of most of the test problems compared to the SDHAM methods.
\begin{table}[htbp]
  \centering
  \caption{List of test problems with references and starting points}
    \begin{tabular}{cll}
    \hline
    \multicolumn{1}{l}{Problems} & Function name & Starting point \bigstrut\\
    \hline
    \multicolumn{3}{l}{Large scale} \bigstrut\\
    \hline
    P1    & Penalty function I \cite{la2004spectral}     & $(1/3,1/3,\cdots,1/3)^T$ \bigstrut[t]\\
    P2    & Trigonometric function \cite{more1978testing}     & $(1/n,\cdots,1/n)^T$ \\
    P3    & Discrete boundary value \cite{more1978testing}     & $(\frac{1}{n+1}(\frac{1}{n+1}-1),\cdots,\frac{1}{n+1}(\frac{n}{n+1}-1))^T$ \\
    P4    & Linear function full rank \cite{more1978testing}    & $(1,1,\cdots,1)^T$ \\
    P5    & Problem 202 \cite{lukvsan2003test}    & $(2,2,\cdots,2)^T$ \\
    P6    & Problem 206 \cite{lukvsan2003test}    & $(1/n,\cdots,1/n)^T$ \\
    P7    & Problem 212 \cite{lukvsan2003test}     & $(0.5,\cdots,0.5)^T$ \\
    P8    & Strictly convex function I \cite{raydan1997barzilai}     & $(1/n,2/n,\cdots,1)^T$ \\
    P9    &  Strictly convex function II \cite{raydan1997barzilai}    & $(1,1,\cdots,1)^T$ \\
    P10   & Brown almost linear \cite{more1978testing}    & $(0.5,\cdots,0.5)^T$ \\
    P11   & Exponential function I \cite{la2004spectral}    & $(\frac{n}{n-1},\cdots,\frac{n}{n-1})^T$ \\
    P12   & Singular function \cite{la2004spectral}    & $(1,1,\cdots,1)^T$ \\
    P13   & Logarithmic function \cite{la2004spectral}    & $(1,1,\cdots,1)^T$ \\
    P14   & Extended Freudenstein and Roth \cite{la2004spectral}   & $(6,3,6,3,\cdots,6,3)^T$ \\
    P15   & Extended Powell singular \cite{la2004spectral}    & $(1.5E-4,\cdots,1.5E-4)^T$ \\
    P16   & Function21 \cite{la2004spectral}    & $(-1,-1,\cdots,-1)^T$ \\
    P17   & Broyden tridiagonal function \cite{more1978testing}    & $(-1,-1,\cdots,-1)^T$ \\
    P18   & Generalized Broyden tridiagonal \cite{lukvsan2003test}    & $(-1,-1,\cdots,-1)^T$ \\
    P19   & Extended Rosenbrock \cite{more1978testing}    & $(-1,1,-1,1,\cdots,-1,1)^T$ \\
    P20   & Extended Himmelblau \cite{momin2013literature}    & $(1,1/n,1,1/n,\cdots,1,1/n)^T$ \\
    P21   & Function 27 \cite{la2004spectral}    & $(100,1/n^2,1/n^2,\cdots,1/n^2)^T$ \\
    P22   & Trigonometric logarithmic function \cite{mohammad2018structured}    & $(1,1,\cdots,1)^T$ \bigstrut[b]\\
    \hline
    \multicolumn{1}{l}{Small scale} &       &  \bigstrut\\
    \hline
    P23   & Bard function \cite{more1978testing}    & $(-1000,-1000,-1000)^T$ \bigstrut[t]\\
    P24   & Brown badly scaled \cite{more1978testing}    & $(1,1)^T$ \\
    P25   & Jennrich and Sampson \cite{more1978testing}    & $(0.2,0.2)^T$ \\
    P26   & Box 3D function \cite{more1978testing}    & $(0,0.1)^T$ \\
    P27   & Rank deficient Jacobian \cite{gonccalves2016local}    & $(-1,1)^T$ \\
    P28   & Rosenbrock function \cite{more1978testing}    & $(-1,1)^T$ \\
    P29   & Parameterized problem \cite{huschens1994use}    & $(10,10)^T$ \\
    P30   & Freudenstein and Roth function \cite{more1978testing}   & $(0.5,-2)^T$ \bigstrut[b]\\
    \hline
    \end{tabular}%
  \label{testproblem}%
\end{table}%

\begin{table}[htbp]
  \centering
  \caption{Numerical results of our ASDH and SDHAM methods for large scale problems $1-22$ with dimension $n=m=1,000$}
  \scalebox{0.8}
  {
    \begin{tabular}{ccccccccccc}
    \toprule
     & \multicolumn{5}{c}{ASDH}              & \multicolumn{5}{c}{SDHAM} \\
    \hline
    Problem & NITER  & NFEVAL & NMVP & TIME  & FVALUE & ITER  & NFEVAL & NMVP & TIME  & FVALUE \\
    \hline
    P1    & 6     & 7     & 19    & 0.022139 & 2.36E-06 & 5     & 6     & 16    & 0.059823 & 1.23E-05 \\
    P2    & 193   & 354   & 580   & 0.26363 & 1.34E-07 & 5     & 12    & 16    & 0.011127 & 1.89E-15 \\
    P3    & 18    & 20    & 55    & 0.063171 & 2.52E-08 & 20    & 23    & 61    & 0.074233 & 2.33E-08 \\
    P4    & 2     & 3     & 7     & 0.02106 & 0.502 & 2     & 3     & 7     & 0.020876 & 0.502 \\
    P5    & 5     & 6     & 16    & 0.013713 & 7.22E-13 & 5     & 6     & 16    & 0.017462 & 1.39E-10 \\
    P6    & 43    & 49    & 130   & 0.099136 & 3E-08 & 44    & 57    & 133   & 0.03711 & 2.96E-08 \\
    P7    & 5     & 6     & 16    & 0.012173 & 9.06E-12 & 7     & 8     & 22    & 0.015851 & 3.32E-12 \\
    P8   & 5     & 6     & 16    & 0.007124 & 500   & 5     & 6     & 16    & 0.01575 & 500 \\
    P9   & 8     & 23    & 25    & 0.039868 & 1669168 & 8     & 23    & 25    & 0.012976 & 1669168 \\
    P10   & 2     & 23    & 7     & 0.005995 & 1.99E-13 & F      & F      & F      &F       & F \\
    P11   & 4     & 5     & 13    & 0.003988 & 9.44E-08 & 5     & 6     & 16    & 0.008688 & 4.53E-08 \\
    P12   & 22    & 41    & 67    & 0.061689 & 55.0298 & 25    & 44    & 76    & 0.081369 & 55.0298 \\
    P13   & 6     & 7     & 19    & 0.025038 & 8.45E-17 & 6     & 7     & 19    & 0.012527 & 1.57E-12 \\
    P14   & 19    & 27    & 58    & 0.038295 & 8.12E-12 & 21    & 30    & 64    & 0.018482 & 6.24E-14 \\
    P15   & 1     & 7     & 4     & 0.018509 & 1.21E-12 & 1     & 7     & 4     & 0.009089 & 1.21E-12 \\
    P16   & 85    & 116   & 256   & 0.13824 & 6.2E-10 & 75    & 101   & 226   & 0.21378 & 2.51E-10 \\
    P17   & 247   & 507   & 742   & 0.26142 & 0.35626 & 19    & 34    & 58    & 0.10981 & 1.63E-11 \\
    P18   & 72    & 125   & 217   & 0.1084 & 0.76757 & 103   & 148   & 310   & 0.13672 & 0.026797 \\
    P19   & 1     & 2     & 4     & 0.009725 & 0     & 1     & 2     & 4     & 0.017676 & 0 \\
    P20   & 17    & 22    & 52    & 0.018535 & 7.54E-13 & 30    & 34    & 91    & 0.088546 & 5.78E-11 \\
    P21   & 17    & 32    & 52    & 0.022589 & 6.37E-07 & 18    & 33    & 55    & 0.064815 & 2.16E-07 \\
    P22   & 6     & 7     & 19    & 0.02308 & 8.16E-17 & 6     & 9     & 19    & 0.29798 & 1.03E-14 \\
    \bottomrule
    \end{tabular}%
    }
  \label{table1}%
\end{table}%

\begin{table}[htbp]
  \centering
  \caption{Numerical results of our ASDH and SDHAM methods for large scale problems $1-22$ with dimension $n=m=5,000$}
  \scalebox{0.8}
  {
    \begin{tabular}{ccccccccccc}
    \hline
     & \multicolumn{5}{c}{ASDH}              & \multicolumn{5}{c}{SDHAM} \bigstrut\\
    \hline
    Problem & NITER  & NFEVAL & NMVP & TIME  & FVALUE & ITER  & NFEVAL & NMVP & TIME  & FVALUE \bigstrut\\
    \hline
    P1    & 15    & 16    & 46    & 0.066157 & 7.93E-05 & 15    & 16    & 46    & 0.17645 & 7.92E-05 \bigstrut[t]\\
    P2    & 120   & 273   & 361   & 0.76396 & 7.91E-08 & F      & F      &  F     & F      & F \\
    P3    & 3     & 5     & 10    & 0.044907 & 1.44E-09 & 3     & 5     & 10    & 0.16634 & 1.44E-09 \\
    P4    & 2     & 3     & 7     & 0.027567 & 0.5004 & 2     & 3     & 7     & 0.09278 & 0.5004 \\
    P5    & 5     & 6     & 16    & 0.029966 & 7.23E-13 & 5     & 6     & 16    & 0.052194 & 6.94E-10 \\
    P6    & 7     & 9     & 22    & 0.060726 & 2.32E-09 & 6     & 8     & 19    & 0.11743 & 2.83E-09 \\
    P7    & 5     & 6     & 16    & 0.018999 & 9.11E-12 & 7     & 8     & 22    & 0.1734 & 3.32E-12 \\
    P8   & 5     & 6     & 16    & 0.075254 & 2500  & 5     & 6     & 16    & 0.047758 & 2500 \\
    P9   & 8     & 27    & 25    & 0.071887 & 2.08E+08 & 9     & 28    & 28    & 0.13815 & 2.08E+08 \\
    P10   & 3     & 29    & 10    & 0.027006 & 5.17E-16 & 2     & 28    & 7     & 0.072295 & NaN \\
    P11   & 3     & 4     & 10    & 0.047036 & 9.54E-08 & 4     & 5     & 13    & 0.046785 & 3.62E-08 \\
    P12   & 27    & 50    & 82    & 0.25832 & 1386.256 & 29    & 52    & 88    & 0.26293 & 1386.255 \\
    P13   & 6     & 7     & 19    & 0.040228 & 3.93E-16 & 6     & 7     & 19    & 0.062598 & 7.6E-12 \\
    P14   & 21    & 29    & 64    & 0.062354 & 1.41E-13 & 21    & 30    & 64    & 0.1584 & 3.12E-13 \\
    P15   & 1     & 7     & 4     & 0.024886 & 6.03E-12 & 1     & 7     & 4     & 0.12042 & 6.03E-12 \\
    P16   & 69    & 96    & 208   & 0.36005 & 5.97E-10 & 83    & 113   & 250   & 0.62034 & 4.57E-10 \\
    P17   & 140   & 305   & 421   & 0.32061 & 2.26E-10 & 19    & 34    & 58    & 0.20367 & 1.63E-11 \\
    P18   & 58    & 98    & 175   & 0.21311 & 2.26E-10 & 97    & 132   & 292   & 0.3257 & 0.031054 \\
    P19   & 1     & 2     & 4     & 0.005155 & 0     & 1     & 2     & 4     & 0.088945 & 0 \\
    P20   & 40    & 50    & 121   & 0.21105 & 1.15E-10 & 30    & 35    & 91    & 0.1997 & 3.18E-11 \\
    P21   & 17    & 32    & 52    & 0.10225 & 6.37E-07 & 18    & 33    & 55    & 0.10118 & 2.16E-07 \\
    P22   & 6     & 7     & 19    & 0.048698 & 3.9E-16 & 6     & 9     & 19    & 0.56375 & 3.33E-12 \bigstrut[b]\\
    \hline
    \end{tabular}%
    }
  \label{table2}%
\end{table}%

\begin{table}[htbp]
  \centering
  \caption{Numerical results of our ASDH and SDHAM methods for large scale problems $1-22$ with dimension $n=m=10,000$}
  \scalebox{0.8}
  {
    \begin{tabular}{ccccccccccc}
    \hline
     & \multicolumn{5}{c}{ASDH}              & \multicolumn{5}{c}{SDHAM} \bigstrut\\
    \hline
    Problem & NITER  & NFEVAL & NMVP & TIME  & FVALUE & ITER  & NFEVAL & NMVP & TIME  & FVALUE \bigstrut\\
    \hline
    P1    & 23    & 24    & 70    & 0.49096 & 0.000142 & 23    & 24    & 70    & 0.13661 & 0.000142 \bigstrut[t]\\
    P2    & 81    & 192   & 244   & 0.67185 & 7.36E-08 &F       & F      & F      & F      & F \\
    P3    & 2     & 4     & 7     & 0.095771 & 1.58E-09 & 2     & 4     & 7     & 0.044607 & 1.58E-09 \\
    P4    & 2     & 3     & 7     & 0.055015 & 0.5002 & 2     & 3     & 7     & 0.030105 & 0.5002 \\
    P5    & 5     & 6     & 16    & 0.06105 & 7.25E-13 & 5     & 6     & 16    & 0.040306 & 1.39E-09 \\
    P6    & 3     & 5     & 10    & 0.020898 & 7.16E-10 & 3     & 5     & 10    & 0.025454 & 7.16E-10 \\
    P7    & 5     & 6     & 16    & 0.040219 & 9.17E-12 & 7     & 8     & 22    & 0.056896 & 3.32E-12 \\
    P8   & 5     & 6     & 16    & 0.1955 & 5000  & 5     & 6     & 16    & 0.051594 & 5000 \\
    P9   & 9     & 30    & 28    & 0.099793 & 1.67E+09 & 10    & 31    & 31    & 0.15435 & 1.67E+09 \\
    P10   & 3     & 4     & 10    & 0.019368 & 4.77E-08 & 3     & 4     & 10    & 0.021593 & 7.24E-08 \\
    P11   & 199   & 3849  & 598   & 4.6179 & 2.49E-15 & 2     & 19    & 7     & 0.025921 & 1.13E-13 \\
    P12   & 28    & 53    & 85    & 0.27666 & 5550.287 & 30    & 55    & 91    & 0.51292 & 5550.286 \\
    P13   & 6     & 7     & 19    & 0.066932 & 7.79E-16 & 6     & 7     & 19    & 0.03482 & 1.51E-11 \\
    P14   & 21    & 29    & 64    & 0.11086 & 2.83E-13 & 21    & 30    & 64    & 0.16018 & 6.24E-13 \\
    P15   & 1     & 7     & 4     & 0.011959 & 1.21E-11 & 1     & 7     & 4     & 0.014576 & 1.21E-11 \\
    P16   & 89    & 122   & 268   & 0.90064 & 5.94E-10 & 86    & 112   & 259   & 0.84402 & 8.47E-10 \\
    P17   & 286   & 620   & 859   & 1.5625 & 0.35626 & 19    & 34    & 58    & 0.20675 & 1.63E-11 \\
    P18   & 56    & 80    & 169   & 0.28569 & 3.2E-10 & 238   & 362   & 715   & 1.3739 & 0.010024 \\
    P19   & 1     & 2     & 4     & 0.016878 & 0     & 1     & 2     & 4     & 0.018502 & 0 \\
    P20   & 24    & 29    & 73    & 0.14056 & 6.17E-11 & 26    & 31    & 79    & 0.23926 & 4.71E-11 \\
    P21   & 17    & 32    & 52    & 0.11498 & 6.37E-07 & 18    & 33    & 55    & 0.17027 & 2.16E-07 \\
    P22   & 6     & 7     & 19    & 0.031825 & 7.77E-16 & 6     & 9     & 19    & 0.12546 & 8.61E-12 \bigstrut[b]\\
    \hline
    \end{tabular}%
    }
  \label{table3}%
\end{table}%

\begin{table}[htbp]
  \centering
  \caption{Numerical results of our ASDH and SDHAM methods for small scale problems $23-30$}
  \scalebox{0.8}
  {
    \begin{tabular}{clcccccccccc}
    \hline
          &       & \multicolumn{5}{c}{ASDH}              & \multicolumn{5}{c}{SDHAM} \bigstrut\\
    \hline
    Problem & \multicolumn{1}{c}{DIM} & ITER  & NFEVAL & NMVP & TIME  & FVALUE & ITER  & NFEVAL & NMVP & TIME  & FVALUE \bigstrut\\
    \hline
    P23   & n=3, m=15 & 1     & 4     & 4     & 0.026898 & 8.7448 & 1     & 4     & 4     & 0.16203 & 8.7448 \bigstrut[t]\\
    P24   & n=2, m=3 & 9     & 12    & 28    & 0.010317 & 7.16E-14 & 10    & 12    & 31    & 0.060814 & 1.05E-27 \\
    P25   & n=2, m=20 & 4     & 20    & 13    & 0.00827 & 62.1811 & 5     & 21    & 16    & 0.046634 & 62.1811 \\
    P26   & n=3, m=10 & 47    & 64    & 142   & 0.063188 & 4.06E-09 & 38    & 50    & 115   & 0.12071 & 1.83E-10 \\
    P27   & n=2, m=3 & 6     & 9     & 19    & 0.008934 & 1.2905 & 6     & 9     & 19    & 0.077812 & 1.2905 \\
    P28   & n=m=2 & 1     & 2     & 4     & 0.005085 & 0     & 1     & 2     & 4     & 0.040352 & 0 \\
    P29   & n=2, m=3 & 6     & 22    & 19    & 0.018706 & 0.49999 & 6     & 22    & 19    & 0.039052 & 0.49999 \\
    P30   & n=m=2 & 81    & 113   & 244   & 0.036083 & 24.4921 & 93    & 133   & 280   & 0.10623 & 24.4921 \bigstrut[b]\\
    \hline
    \end{tabular}%
    }
  \label{table4}%
\end{table}%
\newpage
\begin{figure}[h]
\begin{center}
\includegraphics[angle=0,width=20cm]{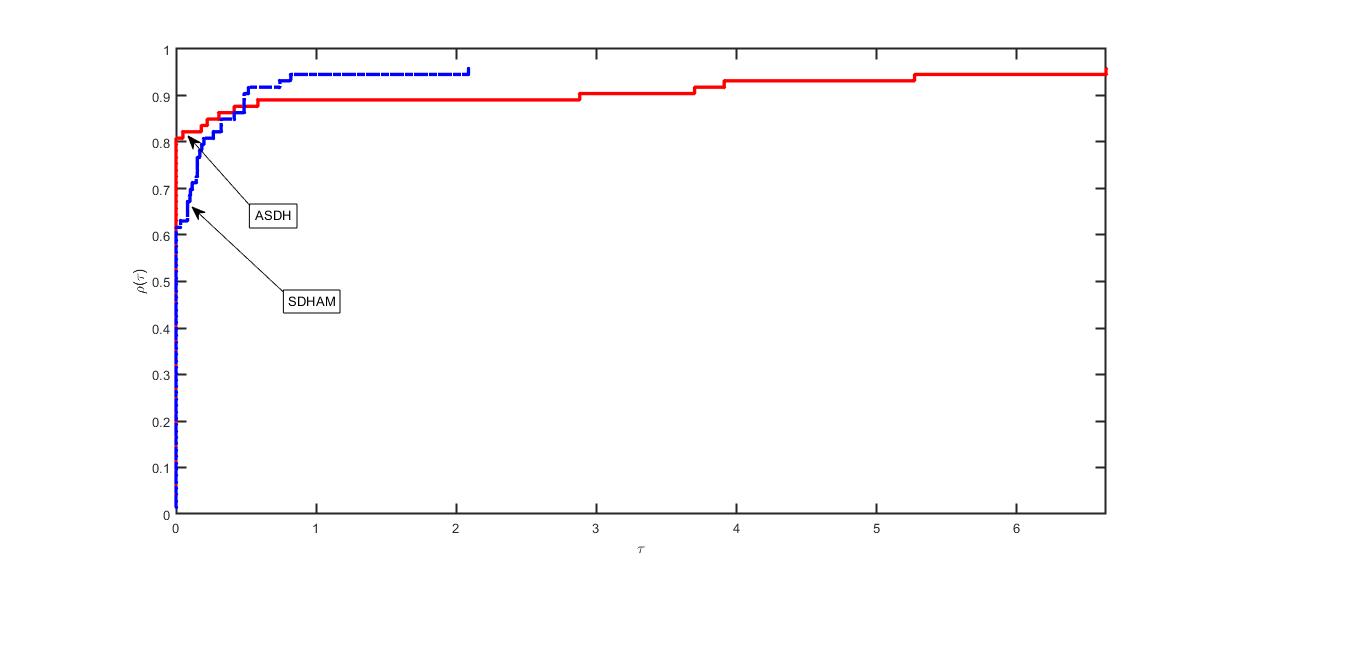}
\end{center}
\vspace{-2cm}{\textbf{\caption{Performance profile with respect to number of iterations }\label{ITER}}}
\end{figure}
\begin{figure}[h]
\begin{center}
\includegraphics[angle=0,width=20cm]{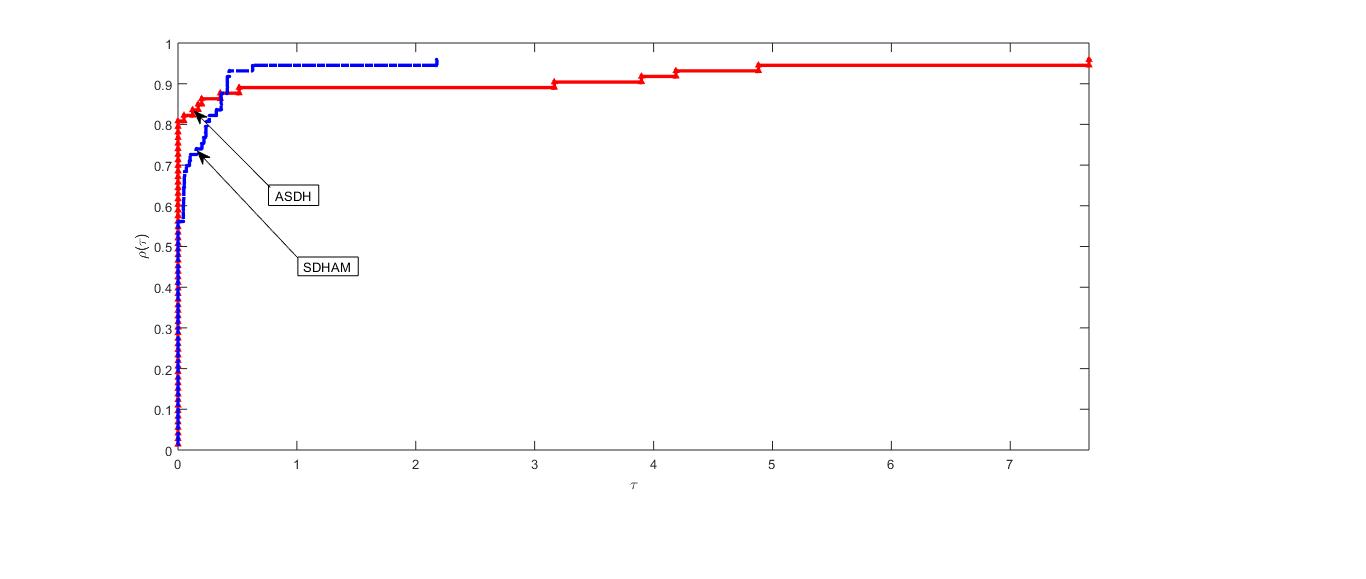}
\end{center}
\vspace{-2cm}{\textbf{\caption{Performance profile with respect to number of function evaluations }\label{FVAL}}}
\end{figure}
\begin{figure}[h]
\begin{center}
\includegraphics[angle=0,width=20cm]{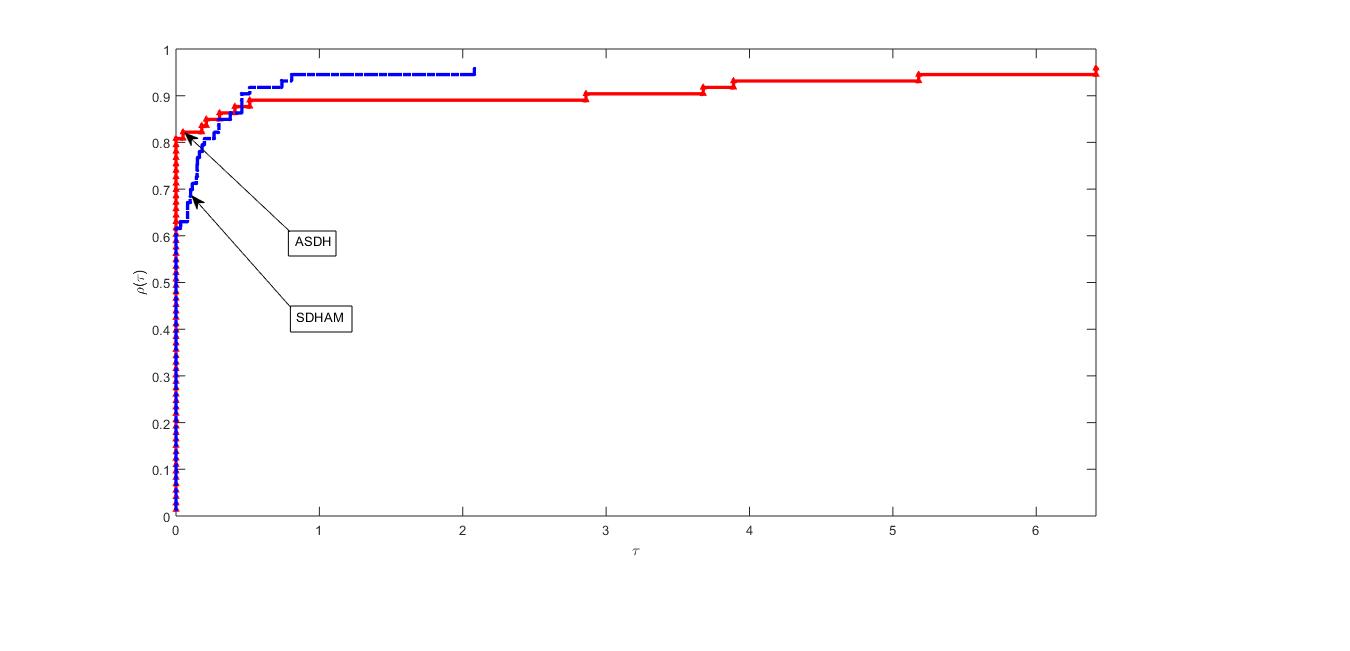}
\end{center}
\vspace{-2cm}{\textbf{\caption{Performance profile with respect to number of matrix-vector product}\label{GVAL}}}
\end{figure}
\begin{figure}[h]
\begin{center}
\includegraphics[angle=0,width=20cm]{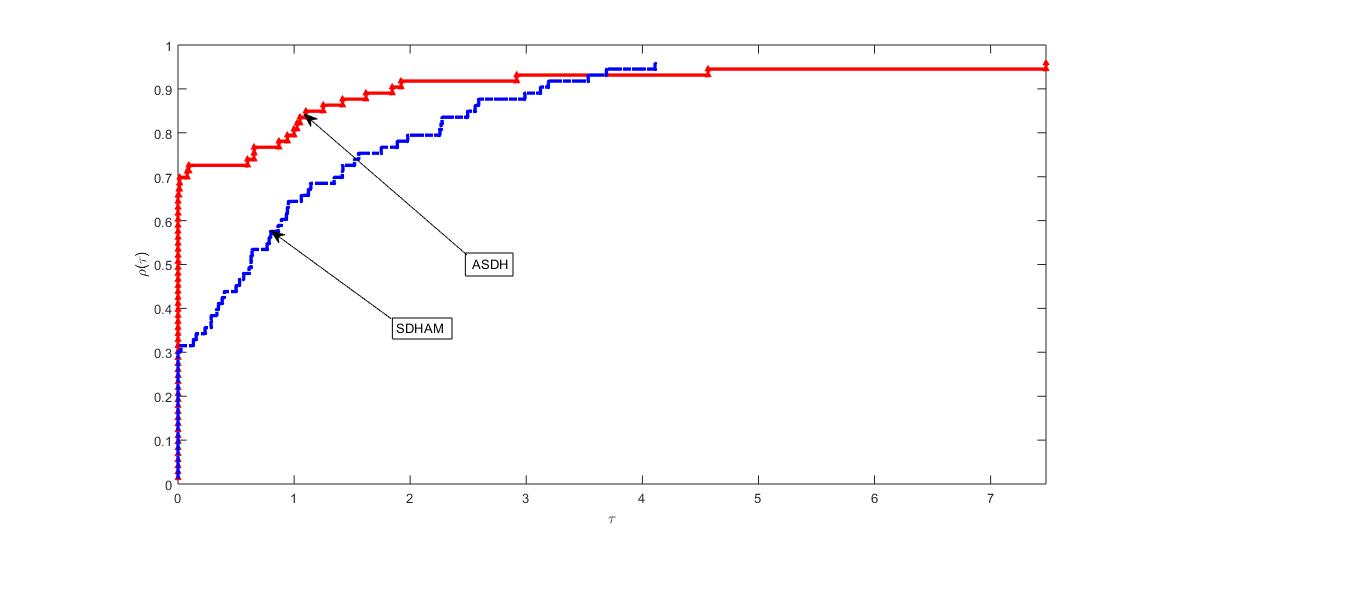}
\end{center}
\vspace{-2cm}{\textbf{\caption{Performance profile with respect to CPU time }\label{TIME}}}
\end{figure}
\newpage
\section{Conclusions}
We have proposed an iterative algorithm with structured diagonal Hessian approximation for solving nonlinear least square problems. The proposed algorithm neither forms nor stores matrices which make it suitable for large scale problems. We have devised appropriate safeguards to ensure the search directions generated by our proposed algorithm are descent. The proposed ASDH method was developed using a diagonal Hessian approximation that contains more information of the objective function than the one proposed by Mohammad and Santos \cite{mohammad2018structured}. We presented some preliminary numerical experiments to show the robustness of our ASDH method.\\\\
\vskip-0.5cm \noindent{\bf Acknowledgements:} 
The first and second authors acknowledge the financial support provided by King Mongkut's University of Technology Thonburi through the "KMUTT 55th Anniversary Commemorative Fund". The first author was supported by the Petchra Pra Jom Klao Doctoral Scholarship Academic for Ph.D. Program at KMUTT. This project was supported by the Theoretical and Computational Science (TaCS) Center under Computational and Applied Science for Smart Innovation Research Cluster (CLASSIC), Faculty of Science, KMUTT.
\newpage
\bibliographystyle{plain}
\bibliography{nlspaper1}

\begin{thebibliography}{10}

\bibitem{barz2015nonlinear}
Tilman Barz, Stefan K{\"o}rkel, G{\"u}nter Wozny, et~al.
\newblock Nonlinear ill-posed problem analysis in model-based parameter
  estimation and experimental design.
\newblock {\em Computers \& Chemical Engineering}, 77:24--42, 2015.

\bibitem{cornelio2011regularized}
Anastasia Cornelio.
\newblock Regularized nonlinear least squares methods for hit position
  reconstruction in small gamma cameras.
\newblock {\em Applied Mathematics and Computation}, 217(12):5589--5595, 2011.

\bibitem{dehghani2018scaled}
R~Dehghani and N~Mahdavi-Amiri.
\newblock Scaled nonlinear conjugate gradient methods for nonlinear least
  squares problems.
\newblock {\em Numerical Algorithms}, pages 1--20, 2018.

\bibitem{dennis1989convergence}
JE~Dennis, H{\'e}ctor~J Martinez, and Richard~A Tapia.
\newblock Convergence theory for the structured bfgs secant method with an
  application to nonlinear least squares.
\newblock {\em Journal of Optimization Theory and Applications},
  61(2):161--178, 1989.

\bibitem{dolan2002benchmarking}
Elizabeth~D Dolan and Jorge~J Mor{\'e}.
\newblock Benchmarking optimization software with performance profiles.
\newblock {\em Mathematical programming}, 91(2):201--213, 2002.

\bibitem{fletcher1987hybrid}
R~Fletcher and C~Xu.
\newblock Hybrid methods for nonlinear least squares.
\newblock {\em IMA Journal of Numerical Analysis}, 7(3):371--389, 1987.

\bibitem{golub2003separable}
Gene Golub and Victor Pereyra.
\newblock Separable nonlinear least squares: the variable projection method and
  its applications.
\newblock {\em Inverse problems}, 19(2):R1, 2003.

\bibitem{gonccalves2016local}
Douglas~S Gon{\c{c}}alves and Sandra~A Santos.
\newblock Local analysis of a spectral correction for the gauss-newton model
  applied to quadratic residual problems.
\newblock {\em Numerical Algorithms}, 73(2):407--431, 2016.

\bibitem{huschens1994use}
J{\"u}rgen Huschens.
\newblock On the use of product structure in secant methods for nonlinear least
  squares problems.
\newblock {\em SIAM Journal on Optimization}, 4(1):108--129, 1994.

\bibitem{kim2007interior}
Seung-Jean Kim, Kwangmoo Koh, Michael Lustig, Stephen Boyd, and Dimitry
  Gorinevsky.
\newblock An interior-point method for large-scale $l_1-$regularized least
  squares.
\newblock {\em IEEE journal of selected topics in signal processing},
  1(4):606--617, 2007.

\bibitem{kobayashi2010nonlinear}
Michiya Kobayashi, Yasushi Narushima, and Hiroshi Yabe.
\newblock Nonlinear conjugate gradient methods with structured secant condition
  for nonlinear least squares problems.
\newblock {\em Journal of computational and applied mathematics},
  234(2):375--397, 2010.

\bibitem{la2004spectral}
William La~Cruz, Jos{\'e}~Mario Mart{\'\i}nez, and Marcos Raydan.
\newblock Spectral residual method without gradient information for solving
  large-scale nonlinear systems: theory and experiments.
\newblock {\em
  $http://kuainasi.ciens.ucv.ve/mraydan/download_papers/TechRep.pdf$}, 2004.

\bibitem{li2012maximum}
Junhong Li, Feng Ding, and Guowei Yang.
\newblock Maximum likelihood least squares identification method for input
  nonlinear finite impulse response moving average systems.
\newblock {\em Mathematical and Computer Modelling}, 55(3-4):442--450, 2012.

\bibitem{lukvsan2003test}
Ladislav Luk{\v{s}}an and Jan Vlcek.
\newblock Test problems for unconstrained optimization.
\newblock {\em Academy of Sciences of the Czech Republic, Institute of Computer
  Science, Technical Report}, (897), 2003.

\bibitem{mohammad2018structured}
Hassan Mohammad and Sandra~A Santos.
\newblock A structured diagonal hessian approximation method with evaluation
  complexity analysis for nonlinear least squares.
\newblock {\em Computational and Applied Mathematics}, 37(5):6619--6653, 2018.

\bibitem{mohammad2019structured}
Hassan Mohammad and Mohammed~Yusuf Waziri.
\newblock Structured two-point stepsize gradient methods for nonlinear least
  squares.
\newblock {\em Journal of Optimization Theory and Applications},
  181(1):298--317, 2019.

\bibitem{mohammad2019survey}
Hassan Mohammad, Mohammed~Yusuf Waziri, and Sandra~Augusta Santos.
\newblock A brief survey of methods for solving nonlinear least-squares
  problems.
\newblock {\em Numerical Algebra, Control \& Optimization}, 9(1):1--13, 2019.

\bibitem{momin2013literature}
Jamil Momin and Yang Xin-She.
\newblock A literature survey of benchmark functions for global optimization
  problems.
\newblock {\em Journal of Mathematical Modelling and Numerical Optimisation},
  4(2):150--194, 2013.

\bibitem{more1978testing}
Jorge~J Mor{\'e}, Burton~S Garbow, and Kenneth~E Hillstrom.
\newblock Testing unconstrained optimization software.
\newblock Technical report, Argonne National Lab., IL (USA), 1978.

\bibitem{raydan1997barzilai}
Marcos Raydan.
\newblock The barzilai and borwein gradient method for the large scale
  unconstrained minimization problem.
\newblock {\em SIAM Journal on Optimization}, 7(1):26--33, 1997.

\bibitem{tang2011regularization}
Li~Min Tang.
\newblock A regularization homotopy iterative method for il--posed nonlinear
  least squares problem and its application.
\newblock In {\em Applied Mechanics and Materials}, volume~90, pages
  3268--3273. Trans Tech Publ, 2011.

\bibitem{xu1990hybrid}
CX~Xu.
\newblock Hybrid method for nonlinear least-square problems without calculating
  derivatives.
\newblock {\em Journal of optimization Theory and Applications},
  65(3):555--574, 1990.

\bibitem{yuan2009subspace}
Ya-Xiang Yuan.
\newblock Subspace methods for large scale nonlinear equations and nonlinear
  least squares.
\newblock {\em Optimization and Engineering}, 10(2):207--218, 2009.

\bibitem{yuan2011recent}
Ya-Xiang Yuan.
\newblock Recent advances in numerical methods for nonlinear equations and
  nonlinear least squares.
\newblock {\em Numerical algebra, control and optimization}, 1(1):15--34, 2011.

\bibitem{zhang2012local}
Hongchao Zhang and Andrew~R Conn.
\newblock On the local convergence of a derivative-free algorithm for
  least-squares minimization.
\newblock {\em Computational optimization and applications}, 51(2):481--507,
  2012.

\bibitem{zhang2010derivative}
Hongchao Zhang, Andrew~R Conn, and Katya Scheinberg.
\newblock A derivative-free algorithm for least-squares minimization.
\newblock {\em SIAM Journal on Optimization}, 20(6):3555--3576, 2010.

\bibitem{zhang2004nonmonotone}
Hongchao Zhang and William~W Hager.
\newblock A nonmonotone line search technique and its application to
  unconstrained optimization.
\newblock {\em SIAM journal on Optimization}, 14(4):1043--1056, 2004.

\end{thebibliography}
%
%
%
%
%
%
%

\end{document}